\documentclass[11pt]{article}
\usepackage{url}
\usepackage{amsmath}
\usepackage{amssymb}
\usepackage{authblk}

\usepackage{amsthm}
\usepackage{empheq}
\usepackage{latexsym}
\usepackage{dsfont}
\usepackage{appendix}
\usepackage{color} 
\usepackage[utf8]{inputenc}
\usepackage[T1]{fontenc}
\usepackage{geometry}
\usepackage{breqn}

\usepackage{eurosym}

\usepackage{graphicx}
\usepackage{caption}
\usepackage{subcaption}

\usepackage{bbm}

\usepackage{titlesec}

\titleformat{\section}
    {\normalfont\Large\bfseries\scshape}{\thesection }{.5em}{\vspace{.5ex}}[\titlerule]
\titleformat*{\subsection}{\large\scshape\bfseries}
\titleformat*{\subsubsection}{\scshape\bfseries}

\newtheorem{Theorem}{Theorem}[part]
\newtheorem{Definition}{Definition}[part]
\newtheorem{Proposition}{Proposition}[part]

\newtheorem{Assumption}{Assumption}[part]
\newtheorem{Lemma}{Lemma}[part]

\newtheorem{Remark}{Remark}[part]
\newtheorem{Example}{Example}[part]

\makeatletter \@addtoreset{equation}{section}

\@addtoreset{Definition}{section}

\@addtoreset{Theorem}{section}

\@addtoreset{Proposition}{section}

\@addtoreset{Property}{section}

\@addtoreset{Assumption}{section}

\@addtoreset{Corollary}{section}

\@addtoreset{Lemma}{section}

\@addtoreset{Remark}{section}

\@addtoreset{Example}{section}

\@addtoreset{Condition}{section}

\def \E{\mathbb{E}}

\def \R{\mathbb{R}}

\allowdisplaybreaks

\begin{document}
\author[1]{Jessica Martin\footnote{jessica.martin@insa-toulouse.fr}}
\affil[1]{INSA de Toulouse\\ IMT UMR CNRS 5219\\ Universit\'e de Toulouse\\ 135 Avenue de Rangueil 31077 Toulouse Cedex 4 France}

\title{The Risk-Sharing problem under limited liability constraints in a single-period model}

\date{\today}

\maketitle

\begin{abstract}
This work provides analysis of a variant of the Risk-Sharing Principal-Agent problem in a single period setting with additional constant lower and upper bounds on the wage paid to the Agent. First the effect of the extra constraints on optimal contract existence is analyzed and leads to  conditions on utilities under which an optimum may be attained. Solution characterization is then provided along with the derivation of a Borch rule for Limited Liability. Finally the CARA utility case is considered and a closed form optimal wage and action are obtained. This allows for analysis of the classical CARA utility and gaussian setting. \end{abstract}

\section{Introduction}

A central question in economics involves understanding the mechanics of incentives. This question has a plethora of applications, from finding an optimal wage structure (fixed wage, bonus schemes etc.) that motivates employees, to understanding the behavior of subcontractors to whom a company may wish to delegate a task. The Principal-Agent problem plays a key role in the analysis of such a question and has attracted increasing interest over the past decades. Such problems allow analysis of the effect of Moral Hazard on contracting situations by providing a model for incentive analysis when there exists a difference in interest between two parties (for example, an employee's primary motivation may not always be maximizing his employer's revenue). This setting has been widely studied with some founding works including the papers by Ross \cite{Ross}, Mirrlees \cite{Mirr}, Hölmstrom \cite{Holm} and Sannikov \cite{Sann}. However, analysis of the Moral Hazard case is often paired with a comparison to the Risk-Sharing benchmark, solution to the less constrained Risk-Sharing problem which models incentive analysis when there is no difference in interest. As well as having less real-life applications than the Moral Hazard problem, this problem is said to be easier to solve and has thus attracted less attention. Some specific literature does nevertheless exist with for example the works of Müller on the Risk-Sharing problem in the exponential utility case \cite{Muller}, or the monograph of Cvitanic and Zhang \cite{CZ}.  \\

Many studies of the Principal-Agent problem allow the Principal to pay the Agent an a priori unbounded wage.  This possiblity seems rather unrealistic. For example the optimal wage in the standard Risk-Sharing problem analyzed by Müller in \cite{Muller} may be negative with positive probability. As a consequence the Agent may work on behalf of the Principal before also owing him money. This setting may also be problematic for the Principal :  if the wage is allowed to be higher than the value of the Principal's wealth process then paying out may make the Principal default. Counteracting this phenomena therefore seems a reasonable extension to the Principal-Agent problem and one way to do so is to enforce bounds on the wage payed by the Principal. This is said to introduce limited liability to the problem : limited liability for the Agent when the wage is bounded below and limited liability for the Principal when the wage is bounded above. \\

Limited liability in the Principal-Agent problem has been introduced before and under different forms, mostly in a Moral Hazard setting. In \cite{Holm} Hölmstrom introduced wage bounds as a by-product of an existence proof for solutions to a Principal-Agent problem. Sappington then analyzed the case of a risk-neutral Principal and a risk-averse Agent, computing the optimal wage for an Agent who chooses his action after the state of nature is realized \cite{Sapp}.  The particular and important case of limited liability and debt contracts was studied by Innes in \cite{Inne} and extended in \cite{Matt} and \cite{Dewa}, with a possibility for negotiation. The risk-neutral Agent case was studied by Park \cite{Park} and Kim \cite{Kim}.  More recently in \cite{Jewi}, Jewitt et al. provided a proof of existence and uniqueness of an optimal action and wage for a risk-neutral Principal and risk-averse Agent under general limited liability bounds and Moral Hazard. The paper provides a characterization of the optimal wage, which is of option form,  but not of the optimal action. In fact an open question brought up by the authors in this paper is the effect of a lower bound on the Agent's action : this is unclear due to the lack of a closed-form expression for the optimum yet it is a vital economic question. Some answers are provided by Kadan et al. in \cite{Kada}. Indeed, in the setting of \cite{Jewi}, the paper provides sufficient conditions under which an increase in a lower wage bound increases the Agent's action. The analysis is followed by numerics which show that the sufficient conditions are too strong. The numerics also raise the question of whether increasing the lower bound can ever decrease the Agent's action as no such examples are found.  \\

The effect of introducing limited liability into the Moral Hazard problem has therefore been quite thoroughly studied in a few different cases with some key open questions remaining. It seems that an equivalent level of analysis of the Risk-Sharing Principal-Agent problem is lacking. Some questions are answered in \cite{Kim} where Kim considers the case of a risk-neutral Principal and risk-averse Agent with a bounded output process. For any fixed action $a$, the author is able to characterize a class of wages that satisfy a set of relevant constraints. Thus rather than fully solving the limited liability problem, Kim proves that for every fixed action there exists a class of wages that satisfy all of the constraints necessary to be a candidate to the optimum. Limited liability in a Risk-Sharing setting is also discussed by Cvitanic and Zhang in \cite{CZ}. Indeed, they derive an implicit characterization of the optimal wage for a Risk-Sharing problem in a continuous setting with a lower bound constraint and for an Agent who has the ability to control volatility. This paper aims to complement these two works. Indeed this study provides insight into the effect of limited liability in a single period setting without volatility control. As is argued in \cite{CZ}, analysis of Risk-Sharing models that do not include volatility control is of interest for comparison with other Principal-Agent models. This is completely in line with the idea that the following work provides a key benchmark : it is a thorough study of the Risk-Sharing counterpart to \cite{Jewi} in a single period case. Furthermore, the setting of this paper allows for an extension beyond the bounded output process and risk-neutral Principal case of \cite{Kim}. In fact both the Principal and the Agent are allowed to express "general" risk-aversion and to be non-CARA and/or non-risk-neutral. Whilst the use of CARA utility  or risk-neutral utility functions increases the tractability of single period models, allowing for more general risk-aversions provides some analysis of the effect of the limited liability constraints that is not linked to the possible idiosyncrasies of specific utilities. As a consequence the analysis of maximizer existence is split from that of its characterization as is described below. Finally, note that we reach a full characterization of both the optimal wage and action, which contrasts with \cite{Kim}. \\

First the question of maximizer existence is tackled using a calculus of variations approach. Two key results are derived :  Theorem \ref{theo:mainex}  and Theorem  \ref{theo:mainMCARA}.  Theorem \ref{theo:mainex} provides general conditions for maximizer existence. In particular it is shown in that under both lower and upper bounds on the admissible wages, existence of solutions can be obtained through light assumptions on the underlying utility functions. However under only a lower bound constraint some extra assumptions on the behavior of the  utilities are required and provided.  
 Theorem \ref{theo:mainMCARA} completes the first result and guarantees existence in the specific lower bound and CARA case using a different "perturbation" type approach. Remark that the setting considered is different from (and complementary to) that of Page's important work on solution existence \cite{Page}. Indeed Risk-Sharing only concerns itself with the Principal's optimization problem, and this work includes analysis of wages that are unbounded above. Also remark that this work differs from the key existence proofs for Principal-Agent problems provided in \cite{Kada0}. Indeed they consider the Principal's problem across the subset of so-called incentive compatible contracts. These contracts verify some optimality properties for the Agent and such optimality helps for topological analysis.  \\
 
With existence established, maximizer characterization is approached using Luenberger's generalized K.K.T. theorem \cite{Luenberger97}. Characterization for a lower bound is provided in Theorem \ref{theo:gencarone}, and for a double limited liability bound in Theorem \ref{theo:gencartwo}. Both of these theorems include a  variant of the Borch rule of Risk-Sharing (\cite{Borch}) for the limited liability problem, establishing that the optimal wage takes on an "option" form. Note that this is coherent with both the work of Jewitt et al. in the Moral Hazard setting \cite{Jewi} and the class of wages obtained by Kim in the Risk-Sharing case in \cite{Kim}. It is also in line with the optimal wage derived in continuous time and under volatility control in \cite{CZ}. This characterization is then illustrated for a variant of the logarithmic utility which emphasizes the wide scope of the results. The final part of the paper builds on these two theorems and deals with the the CARA utility case for which a very tractable optimal wage and action are provided. The specific case of wage positivity in a CARA/gaussian setting is analyzed at the end.  \\

The rest of the document is structured as follows. In Section \ref{sec:pres}, the setting and the limited liability Risk-Sharing Principal-Agent problem are both presented. Section \ref{sec:ex} then deals with maximizer existence. In Section \ref{sec:charac}, characterization of the related optima is presented in general settings with an example.  Finally in Section \ref{sec:CARA} the CARA utility case is analyzed with an emphasis on the CARA and gaussian setting. 

\section{A Risk-Sharing problem with limited liability}
\label{sec:pres}

Consider a basic single period Principal-Agent problem where a Principal employs an an Agent to perform a task in exchange for some compensation. The Agent's action is modeled as some real number $a$. This action affects the Principal's random production process. To introduce it, first let $\Omega$ be an arbitrary uncountable sample space and $(\Omega, \mathbb{F}, \mathbb{P})$ be a related probability space and denote as $L^2(\Omega)$ the set of square integrable random variables. Let some $B$ in $L^2(\Omega)$\footnote{Of course for the Risk-Sharing problem (that follows) to be well-posed, the random variable $B$ must be compatible with the utility functions.}, and denote as $X^a$ the Principal's production process, of initial value $x_0 \in \mathbb{R}$ : 
$$ X^a := x_0 + a + B.$$
Performing the action costs the Agent some effort. Here consider quadratic effort is considered and modeled by $\kappa(a) := K \frac{a^2}{2}$ for $K > 0$ some fixed parameter. As a reward for his effort the Agent is paid a wage denoted $W$ and belonging to $L^2(\Omega)$. In such a context, one may analyze the Risk-Sharing Principal-Agent problem or "first-best" problem : it involves finding the optimal wage and action for the Principal whilst ensuring that the Agent is satisfied with the situation.\footnote{This "one-sided" optimization for the Principal may, with good reason, seem unfair. This problem is in fact a benchmark problem that is often used as a measure of comparison with other Principal-Agent problems such as Moral Hazard where optimization for the Agent also comes into play. One may also note that in our increasingly digitalized economies with an ever growing use of machines, analysis of optimal contracting without Moral Hazard (and thus in a Risk-Sharing setting) is increasingly relevant in itself too.} Let $U_P$ and $U_A$ be the Principal's and the Agent's utility functions verifying $U_P' > 0, U_P'' \leq 0, U_A' > 0$ and $U_A'' \leq 0$, continuous on $\mathbb{R}$ with $\lim_{x \rightarrow -\infty} U(x) = -\infty$. The Risk-Sharing problem involves solving the following optimization problem : 
\begin{equation}
\label{eq:pbRS}
\sup_{(W,a) \in L^2(\Omega) \times \R} \quad \E\left[U_P(X^a -W) \right],
\end{equation}
such that the Agent's participation constraint is satisfied :  
\begin{equation}
\label{eq:pbcons}
\E\left[ U_A (W-\kappa(a)) \right] \geq U_A(y), \quad y \in \mathbb{R^+ }  \; \text{(fixed)}. \\
\end{equation}
This problem has been studied in a number of cases, notably in the CARA (constant absolute risk aversion) utility case where the Principal and Agent are risk-averse. Mathematically consider two risk-aversion constants $\gamma_P > 0$ and $\gamma_A > 0$ and  the CARA utility functions~: 
$$ U_P(x) := - e^{-\gamma_P x} \quad \text{and} \quad U_A(x) := - e^{-\gamma_A x}. $$
Then there exists a unique maximizer $(W^*, a^*)$ for Problem (\ref{eq:pbRS})-(\ref{eq:pbcons}) : 
$$ a^* := \frac 1 K \quad \text{and} \quad W^* := \frac{\gamma_P}{\gamma_P + \gamma_A} X^{a^*} + \beta^*,$$
where $$\beta^* := \inf_{b \in \R} \left\{ \E\left[U_A\left( \frac{\gamma_P}{\gamma_P + \gamma_A} X^{a^*} + b - \kappa(a^*) \right)\right] = U_A(y) \right\}.$$
The reader may find a proof of this result in \cite{CZ} with $B$ following a standard normal distribution, and in \cite{Mart} with $B$ a more general random variable. Now consider for a moment the case where $B$ follows a standard normal distribution (i.e. $B \sim \mathcal{N}(0,1)$ ). Then the action that the Agent should perform is positive, yet the wage that he may receive in return is negative with positive probability... The situation may also be problematic for the Principal : $W^*$ may be greater than $X^{a^*}$ with positive probability too. These observations were key to motivating the work presented in the following sections : it aims at introducing bounds on the values that the wage may take and thus reduce the liability of the Principal and the Agent. The problem becomes a Risk-Sharing Principal-Agent problem with limited liability and writes as : 
\begin{equation}
\label{eq:pbRSx}
\sup_{(W,a) \in L^2(\Omega) \times \R} \quad \E\left[U_P(X^a -W) \right],
\end{equation}
such that :  
\begin{equation}
\label{eq:pbconsx}
\E\left[ U_A (W-\kappa(a)) \right] \geq U_A(y), \quad y \in \mathbb{R^+ } \text{(fixed)} \\
\end{equation}
and either : 
\begin{equation}
\label{eq:Wlower}
m \leq W \quad \mathbb{P}-a.s.,
\end{equation}
or :
\begin{equation}
\label{eq:Wlowerupper}
m \leq W \leq M \quad \mathbb{P}-a.s., 
\end{equation} 
where $m$ and $M$ are two fixed positive parameters with $m \leq y \leq M$ (of course when only considering one-sided limited liability, only $m \leq y$ is assumed). Such an assumption ensures that the constant contract $(y,0)$ satisfies both the participation constraint (\ref{eq:pbconsx}) and the relevant limited liability constraint : either (\ref{eq:Wlower}) or (\ref{eq:Wlowerupper}). It also has economic meaning as for example the upper bound for $y$ assures the Principal that he is not over investing in human capital. With the assumption in mind, two variants of the same problem are thus considered : (\ref{eq:pbRSx})-(\ref{eq:pbconsx})-(\ref{eq:Wlower}) where limited liability for the Agent is enforced, and  (\ref{eq:pbRSx})-(\ref{eq:pbconsx})-(\ref{eq:Wlowerupper}) where limited liability for the Principal is also enforced.  

\begin{Remark}
A very relevant real life application of constraint (\ref{eq:Wlowerupper}) is the effect of pay bands on optimal contracting. Indeed, using a pay band based compensation scheme is a widely used practice amongst organizations and the results of this paper show how a classical linear wage, combining a fixed part and a performance related part, and a classical action may be optimally tweaked in order to fit into a pay band system. 
\end{Remark}

The most widely studied cases for such Principal-Agent problems involve either a risk-neutral Principal with a risk-averse Agent, or as previously illustrated a risk-averse Principal with a risk-averse Agent each under CARA utilities. In both cases, the forms of the utility functions simplify the reasoning required to analyze the optimal wage and action. In fact in the case of a risk-averse Agent with a risk-neutral Principal, solving the lower bounded Limited Liability problem for example is immediate. Indeed the contract defined by :
$$ W^* = y + \kappa(a^*) \quad \text{and} \quad a^* = \frac{1}{K},$$
is optimal. 
In the CARA case one may find necessary and sufficient optimality conditions through a Lagrangian taking advantage of explicit calculations. Furthermore in the Risk-Sharing and CARA case, the author previously proved that maximizer existence was intricately linked to the pleasing form of the exponential functions (see \cite{Mart}). Another limit of the CARA (constant absolute risk aversion) utility functions is that they are blind to the effect of wealth. As a consequence, an aim is to take this analysis of the Risk-Sharing Limited Liability problem beyond the risk-neutral / CARA utility settings. Here are some examples of utility functions that one may wish to consider: 
\begin{itemize}
\item[-] A $C^2$ extension of the logarithmic utility : 
\begin{equation}
\label{eq:utilog}U(x) = log(x)\textbf{1}_{x \geq 1} - \frac12(x^2 - 4x + 3)\textbf{1}_{x < 1}.
\end{equation}
\item[-] A partially IARA and $C^2$ variant of the CARA utility~: 
\begin{equation}
\label{eq:utiI}
 U(x) = -e^{-x}\textbf{1}_{x \geq 0} - \left(\frac12 x^2 - x + 1\right)\textbf{1}_{x < 0}, \quad \text{for} \; \gamma \; \text{fixed}.
 \end{equation}
 This utility partially disrupts the constant risk aversion of the standard CARA utility. Indeed, $x \mapsto \frac{U''(x)}{U'(x)}$ is increasing up to 0 and is constant and worth $1$ beyond $0$. On $\mathbb{R}^-$ it thus exposes "IARA" (increasing absolute risk aversion). 
 \item[-] A $C^3$ utility, with $U'''(x) \geq 0$~:
 \begin{equation}
 \label{eq:utid}
 U(x) = \arctan(x)\textbf{1}_{x \geq \frac{1}{\sqrt{3}}} + \frac{1}{16} \left\{ -3\sqrt{3} x^2 + 18x + \frac{16\pi}{6} + \sqrt{3} - \frac{18}{\sqrt{3}}\right\} \textbf{1}_{x <\frac{1}{\sqrt{3}}}.
  \end{equation}
  In particular this utility is partially "DARA" (decreasing risk aversion) and partially "IARA". Indeed, $x \mapsto-\frac{U''(x)}{U'(x)}$ is increasing up to $\frac{1}{\sqrt{3}}$ and decreasing beyond.   %\item A partially "DARA" (decreasing absolute risk aversion)  $C^2$ variant of the CARA utility~: 
% \begin{equation}
% \label{eq:utid}
% U(x) = -\frac{e^{-x}}{2} \textbf{1}_{x \leq 0} + \arctan(x)\textbf{1}_{x > 0} + \frac{1}{2} - \frac{\pi}{4}.
% \end{equation}
\end{itemize}

Assume from now on the following utility/production process compatibility assumption. 

\begin{Assumption}
The constant contract $(y,0)$ provides the Principal with a finite utility : 
$$ \E\left[ |U_P(X^0-y) |\right] < + \infty.$$
In other words, the random variable $B$ and the Principal's utility are compatible. 
\end{Assumption}

Note that discussions on the validity all three of these utility functions are beyond the scope of this work. However the first two are variants on classical utilities whose validity has been extensively studied. The third one is upper-bounded which is useful for obtaining upper-semi-continuity of its related expected utility. As a whole they introduce interesting properties such as variations on risk aversion into the problem, and they have slightly less tractable expressions than CARA utilities : they do not lend themselves as easily to full analysis through a Lagrangian. As a consequence a first step involves a calculus of variation approach to provide conditions that ensure maximizer existence.

\section{Results on maximizer existence}
\label{sec:ex}
 
 In the following, some key elements of theory for infinite dimension optimization are given. The reader may find more information in \cite{Kurdi}. \\

 The limited liability problems involve optimizing across two subsets of the Hilbert space : 
$$ E :=L^2(\Omega) \times \mathbb{R},$$
with the scalar product : 
$$ \left \langle \begin{pmatrix} 
W_1 \\
a_1 
\end{pmatrix} ;   \begin{pmatrix} 
W_2 \\
a_2
\end{pmatrix} \right \rangle_{E} := \mathbb{E}[W_1 W_2] +  a_1 a_2, 
$$ 
and the norm : 
$$ ||(W,a)||_{E} := \sqrt{\mathbb{E}[W^2] + |a|^2 }. $$
Let $(W_1, a_1)$ and $(W_2, a_2)$ be two points in E. Denote the Euclidean distance on $E$ as $d_E$ : 
$$ d_E(W_1, a_1), (W_2, a_2)) :=  ||(W_1-W_2,a_1-a_2)||_{E}. $$

\begin{Definition}
A functional $f : E \rightarrow \bar{\mathbb{R}}$ is upper semi continuous at $(W, a)$ in $E$ if for any sequence $(W_n, a_n)_{n \in \mathbb{N}}$ such that : 
$$ d((W_n, a_n), (W_,a))_E \underset{n \rightarrow + \infty} {\rightarrow} 0,$$
it holds that :
\begin{equation} 
\label{eq:usc}
 f(W,a) \geq  \underset{n \rightarrow + \infty} {\limsup} \; f(W_n, a_n).
 \end{equation}
 It is weakly upper semi continuous if the limits in \ref{eq:usc} are taken to be weak-limits. 
 \end{Definition}

With this definition in mind, the following two theorems provide conditions that ensure existence of solutions to optimisation problems on Hilbert spaces. These will be applied to the limited liability problem further on. 

\begin{Theorem}
\label{theo:exibound}
Suppose that $f : M \subseteq \bar{\mathbb{R}}$ is weakly upper semi continuous over the bounded and weakly closed subset $M$ of $E$. Then, 
$$ \exists (W_0, a_0) \in M, \quad \text{such that} \quad f(W_0,a_0) = \sup_{(W,a) \in M} f(W,a).$$
\end{Theorem}

\begin{Theorem}
\label{theo:exicoer}
Suppose that $f : M \subseteq \bar{\mathbb{R}}$ is coercive and weakly upper semi continuous over the weakly closed subset $M$ of $E$. Then, 
$$ \exists (W_0, a_0) \in M, \quad \text{such that} \quad f(W_0,a_0) = \sup_{(W,a) \in M} f(W,a).$$

\end{Theorem}
 
\subsection{Preliminary results on the model}

Before applying Theorems \ref{theo:exibound} and \ref{theo:exicoer} to obtain existence results,  some important results on the underlying model are established. First denote the subset of admissible contracts for Problem (\ref{eq:pbRS})-(\ref{eq:pbcons})-(\ref{eq:Wlower}) as $C_m^M$  :
$$ C_m^M := \Big\{ (W,a) \in E, \quad m \leq W \leq M \; \; \mathbb{P}\text{-a.s.},  \quad \mathbb{E}\left[U_A(W-\kappa(a)) \right] \geq U_A(y) \Big\},$$
and for Problem (\ref{eq:pbRS})-(\ref{eq:pbcons})-(\ref{eq:Wlowerupper}) as $C_m$ : 
$$ C_m := \Big\{ (W,a) \in E,  \quad m \leq W \; \; \mathbb{P}\text{-a.s.},  \quad \mathbb{E}\left[U_A(W-\kappa(a)) \right] \geq U_A(y) \Big\}.$$

These sets denote the sets of contracts in $L^2(\Omega) \times \mathbb{R}$ that satisfy both the participation constraint and the limited liability constraint. As $m \leq y \leq M$, they both contain the constant contract $(y,0)$ and are thus non-empty.  \\

The limited liability optimization problems may be rewritten as follows~:
\begin{equation}
\label{eq:prob1}
\sup_{(W,a) \in C_m^M} \quad \E\left[U_P(X^a -W) \right], 
\end{equation}
and 
\begin{equation}
\label{eq:prob2}
\sup_{(W,a) \in C_m } \quad \E\left[U_P(X^a -W) \right]. 
\end{equation}

The following remarks discuss on the possible values of the utility functions. 

\begin{Remark}
\label{rem:up}
For any $(W,a)$ in $C_{m}^M$ or $C_m$, using Jensen's inequality (as $U_A$ is concave) and as $U_A$ is increasing : 
$$ \E[U_A(W - \kappa(a))] \geq U_A(y) \Rightarrow \E[W] \geq y + \kappa(a).$$ 
Therefore (again through Jensen's inequality for $U_P$ and as $U_P$ is increasing)~: 
$$ \E[U_P(X^a - W)] \leq U_P(x_0 + \E[B] + a - \E[W]) \leq U_P(x_0 - y + \E[B] + a^* - \kappa\left(a^*\right)),$$
where $a^* = \frac1K = \text{argsup}_{x \in \mathbb{R}} x - \kappa(x),$
and the Principal's utility is upper bounded across both constraint sets. 
\end{Remark}

\begin{Remark}
\label{rem:ua}
For any $(W,a) \in E$, $\E[U_A(W-\kappa(a))] < + \infty.$ Indeed, as $W$ belongs to $L^2(\Omega)$, $\E[W]-\kappa(a)$ is some finite number. Applying Jensen's inequality, as $U_A$ is concave~: 
$$ \E[U_A(W) - \kappa(a)] \leq  U_A( \E[W]-\kappa(a)) < + \infty.$$ As the participation constraint also lower bounds the Agent's utility, one may deduce that for any $(W,a)$ in $C_m^M$ or $C_m$, $\E[U_A(W-\kappa(a))]$ exists. 
\end{Remark}

In the following, some preliminary results are given. A few hold for both $C_m$ and $C_m^M$ - the notation $\mathcal{C}$ is used when one subset can be trivially substituted for the other.

\begin{Remark}
\label{rem:pos}
For any $(W,a)$ in $\mathcal{C}$ where $a$ belongs to $\R_+,$ $(W, |a|)$ also belongs to $\mathcal{C}$ through the symmetry of the quadratic cost $\kappa$. Furthermore, 
$$ U_P(X^{a}-W) \leq U_P(X^{|a|}-W) \quad \mathbb{P}-a.s.$$
thus an optimal action (if it exists) will be non-negative. 
\end{Remark}

This remark is sound as it makes sense for the Principal to align the effect of the Agent's action with positive variations of the production process.  It also means that from now on and without lack of generality, one can consider that the Agent's actions take their values in $\R^+$.

\begin{Lemma}
\label{lem:coerpos}
The Principal's value function is coercive in $a$ across  $\mathcal{C}.$ 
\end{Lemma}

\begin{proof}
Consider a sequence $(W_n, a_n)_{n \in \mathbb{N}}$ in $\mathcal{C}$ such that $a_n \underset{n \rightarrow + \infty}{\rightarrow}  + \infty. $ Applying Jensen's inequality to the Participation Constraint (\ref{eq:pbcons})  :
$$ \E[W_n] \geq y + \kappa(a_n).$$
Now applying Jensen to the Principal's value function coercivity is obtained : 
$$ \E\left[U_P(X^{a_n} -W_n ) \right] \leq U_P(x_0 + a_n + \E[B] - y - \kappa(a_n)) \underset{n \rightarrow + \infty}{\rightarrow} - \infty.$$
\end{proof}

\begin{Lemma}
\label{lemma:closed}
Let $(W_n, a_n)_{n \in \mathbb{N}}$ in $\mathcal{C}$ be such that : 
$$ (W_n, a_n) \underset{n \rightarrow + \infty}{\rightharpoonup} (W, a) \; \in E,$$
then $(W,a) \in \mathcal{C}$. In other words, $\mathcal{C}$ is a weakly closed subset of $E$. 
\end{Lemma}
\begin{proof}
Let $$ (W_n, a_n) \underset{n \rightarrow + \infty}{\rightharpoonup} (W, a) \; \in E.$$
\begin{enumerate}
\item
First one may prove that $W \geq m$ \; $\mathbb{P}-a.s.$ (the reasoning trivially extends to the upper bound in $C_m^M$). To do so suppose that $\mathbb{P}(W < m) > 0$ .

The weak-convergence of $(W_n, a_n)$ implies in particular that \begin{equation*}
\E[W_n \Phi ]  \underset{n \rightarrow + \infty}{\longrightarrow} \E[W \Phi ] \quad  \; \forall \Phi \in L^2(\Omega),
\end{equation*}
and thus : 
\begin{equation}
\label{eq:conv}\E[(W_n - m) \Phi ]  \underset{n \rightarrow + \infty}{\longrightarrow} \E[(W -m) \Phi ] \quad  \; \forall \Phi \in L^2(\Omega).
\end{equation}
Now set $\Phi = \textbf{1}_{W  < m }$.  Then $\Phi$ belongs to $L^2(\Omega)$ and using Equation (\ref{eq:conv}) :  
$$\E[(W_n - m) \textbf{1}_{W  < m }]  \underset{n \rightarrow + \infty}{\longrightarrow} \E[(W - m) \textbf{1}_{W  < m }].$$
By construction, $\E[(W_n - m) \textbf{1}_{W  < m }] \geq 0$ and $ \E[(W - m) \textbf{1}_{W  < m }]\leq 0$ and so the convergence cannot hold. Thus $W \geq m \; \mathbb{P}-a.s..$
\item Now it remains to deal with the participation constraint. As $ (W_n, a_n) \underset{n \rightarrow + \infty}{\rightharpoonup} (W, a) \; \in E$ and $\kappa$ is a continuous mapping on $\mathbb{R}$, it holds that : 
$$ (W_n, \kappa(a_n)) \underset{n \rightarrow + \infty}{\rightharpoonup} (W, \kappa(a)).$$
In particular : 
\begin{equation*}
\left\langle\begin{pmatrix}
W_n \\
\kappa(a_n) 
\end{pmatrix}, 
\begin{pmatrix}
1\\
-1
\end{pmatrix} 
\right\rangle
  \underset{n \rightarrow + \infty}{\longrightarrow}
\left\langle\begin{pmatrix}
W \\
\kappa(a) 
\end{pmatrix}, 
\begin{pmatrix}
1\\
-1
\end{pmatrix} 
\right \rangle,
\end{equation*}
and thus : 
\begin{equation}
\label{eq:col}
\mathbb{E} \left[ (W_n - W - \kappa(a_n) + \kappa(a))^2 \right]   \underset{n \rightarrow + \infty}{\longrightarrow}0.
\end{equation}
Now as $U_A'' \leq 0 $ meaning that the Agent's utility is concave, the following inequality holds for any $n \in \mathbb{N}$: 
 \begin{align*}
\E[U_A(W_n - \kappa(a_n))] \leq \E[U_A(W - \kappa(a))] + \E[U_A'(W-\kappa(a))(W_n - W - \kappa(a_n) + \kappa(a))],
\end{align*}
and through Cauchy-Schwarz : 
$$ \E[U_A(W_n - \kappa(a_n))] \leq \E[U_A(W - \kappa(a))] + \E\left[{U_A'}^2(W-\kappa(a))\right]  \E\left[(W_n - W - \kappa(a_n) + \kappa(a))^2\right].$$
As $U_A' \geq 0$, $U_A'' \leq 0$ and $W \geq m$ : 
$$ 0 \leq \E\left[{U_A'}^2(W - \kappa(a))\right] \leq  {U_A'}^2(m - \kappa(a)) \in \mathbb{R}^+.$$
Using (\ref{eq:col}) one may deduce that : 
$$ \E\left[{U_A'}^2(W-\kappa(a))\right]  \E\left[(W_n - W - \kappa(a_n) + \kappa(a))^2\right] \underset{n \rightarrow + \infty}{\longrightarrow}0,$$
and thus $$ \limsup_{n \rightarrow + \infty} \E[U_A(W_n - \kappa(a_n))] \leq  \E[U_A(W - \kappa(a))] .$$
In particular, as $U_A(y) \leq \E[U_A(W_n - \kappa(a_n))]$ it follows that :  $$ U_A(y) \leq \E[U_A(W - \kappa(a))].$$
\end{enumerate}
Therefore $\mathcal{C}$ is a weakly closed subset of $E$. 
\end{proof}

Note that the proof of Lemma \ref{lemma:closed} includes proof that for any sequence $(W_n, a_n)$ in $\mathcal{C}$ such that
$$ (W_n, a_n) \underset{n \rightarrow + \infty}{\rightharpoonup} (W,a), $$
it holds that : 
$$ \limsup_{n \rightarrow + \infty} \E[U_A(W_n - \kappa(a_n))] \leq  \E[U_A(W - \kappa(a))] .$$
Thus the Agent's expected utility is weakly upper semi-continuous through construction of the constraint sets. \\

The following result will be useful further on when characterizing the maximizers.

\begin{Lemma}
\label{lem:conv}
$\mathcal{C}$ is a convex subset of $E$. 
\end{Lemma} 
\begin{proof}
Let $(W_1, a_1)$ and $(W_2, a_2)$ belong to $\mathcal{C}$ and let $\lambda > 0. $ Then : 
$$ m \leq \lambda W_1 + (1 - \lambda)W_2  \quad \mathbb{P}-a.s..$$
This trivially extends to the upper bound in $C_m^M$.
It remains to deal with the Participation Constraint : 
\begin{align*}
&\E[U_A(\lambda W_1 + (1 - \lambda)W_2  - \kappa( \lambda a_1 + (1 - \lambda)a_2 ) )]\\
&\geq \E[U_A(\lambda W_1 + (1 - \lambda)W_2  - \lambda \kappa( a_1) + (1 - \lambda)\kappa(a_2 ) )]\\
&\geq \lambda \E[U_A( W_1  -  \kappa( a_1) )] + (1-\lambda )\E[U_A( W_2  -  \kappa( a_2) )] \geq U_A(y).
\end{align*}
Thus $(\lambda W_1 + (1 - \lambda)W_2, \lambda a_1 + (1 - \lambda)a_2)$ belongs to $\mathcal{C}$ and convexity holds. \\
\end{proof}

\begin{Remark}
Mazur's theorem may be applied to deduce from Lemma \ref{lem:conv} and Lemma \ref{lemma:closed} that $\mathcal{C}$ is also strongly closed. 
\end{Remark}

\begin{Lemma}
\label{lemma:concave}
The Principal's value function is concave.
\end{Lemma}
\begin{proof}
Let $(W_1, a_1)$ and $(W_2, a_2)$ belong to $\mathcal{C}$ and let $\lambda > 0. $ Then : 
\begin{align*}
&\E[U_P(X^{\lambda a_1 + (1-\lambda) a_2} - \lambda W_1  - (1-\lambda)W_2 )]\\
& = \E[U_P(\lambda X^{a_1} - \lambda W_1 + (1-\lambda)X^{a_2} + (1 - \lambda)W_2 )]\\
&\geq \lambda \E[U_P( X^{a_1}  - W_1   )] + (1-\lambda )\E[U_P(X^{a_2}  -  W_2)].
\end{align*}

\end{proof}

The following result provides sufficient conditions for upper semi-continuity of the Principal's value function across $\mathcal{C}$.

\begin{Lemma}
\label{lemma:usc}
Suppose that either $B$ is upper bounded : 
$$ B \leq b_{max}, \quad \mathbb{P}-a.s.$$
for some real value $b_{max}$ or that $U_P$ is upper bounded. Then the Principal's value function is strongly upper semi-continuous across $\mathcal{C}$.
\end{Lemma}
\begin{proof}
Let $(W_n, a_n)$ be a sequence in $\mathcal{C}$ that converges strongly to $(W,a)$. As $\mathcal{C}$ is closed, (W,a) belongs to $\mathcal{C}$ too. Up to a subsequence, 
$$ U_P(X^{a_n} - W_n) \underset{n \rightarrow + \infty}{\longrightarrow} U_P(X^a - W), \quad \mathbb{P}-a.s..$$
Now, if $U_P$ is upper bounded then $U_P(X^{a_n} - W_n) \leq U_{max}  \quad \mathbb{P}-a.s., $ for some real $U_{max}$. If $B$ is upper bounded then :
$$ U_P(X^{a_n} - W_n) \leq U_P(x_0 + {a_{max}} + b_{max} - m)  \quad \mathbb{P}-a.s..$$
Applying Fatou's lemma one may conclude : 
$$ \limsup_{n \rightarrow + \infty} \E[U_P(X^{a_n} - W_n)] \leq \E[U_P(X^a-W)].$$
\end{proof}

These sufficient conditions for upper semi continuity are probably suboptimal. However, they allow the analysis of the limited liability problem to be applied in interesting settings beyond the CARA utility case or risk/neutral case. For example the Principal's expected utility obtained through the logarithmic utility defined by (\ref{eq:utilog}) and any upper bounded random variable will be upper semi continuous. Note that obtaining necessary conditions for upper semi continuity functions is a topic of research in itself. \\

Finally, the following result is related to the coercivity in $W$ of the Principal's value function.

\begin{Lemma}
\label{lem:coer}
Suppose that 
$$ \text{supp}(B) \subset [b_1, b_2], \quad (b_1, b_2) \in \R^2. $$
Suppose that $U_P''' \geq 0$. Then for any  $0 \leq a \leq a_{max}$ where $ a_{max}$ some positive constant, the Principal's value function is coercive in $W$ across $C_m$. 
\end{Lemma}
\begin{proof} 
Let $(W_n, a_n)_n$ be a sequence in $C_m$ such that $0 \leq a_n \leq a_{max}$ and $\E[W_n^2] \rightarrow + \infty.$ One may consider two cases. \\

- If  $\E[W_n] \rightarrow + \infty,$ then through Jensen's inequality applied to the Principal's value function~: 
$$ \E\left[U_P(X^{a_n} -W_n ) \right] \leq U_P(x_0 + a_n + \E[B] -\E[W_n]) \underset{n \rightarrow + \infty}{\rightarrow} - \infty.\\$$
- If not,  suppose $\E[W_n] \leq K$ with $K$ some real constant. Through a Taylor expansion : 
\begin{align*} 
\E\left[U_P(X^{a_n} - W_n)\right] &= \E[U_P(X^{0} - y)] +  \E[U_P'(X^{0} - y)(a_n - W_n + y)]\\
&+  \E[U_P''(X^{a_\epsilon} - W_\epsilon)(a_n - W_n + y)^2],
\end{align*}
where $(W_\epsilon, a_\epsilon)$ is some convex combination of $(W_n, a_n)$ and $(y, 0)$ (and belongs to $C_m$ through Lemma \ref{lem:conv}
). Now as $U_P' > 0,$ it holds that : 
$$ \E[U_P'(X^{0} - y)(a_n - W_n + y)] \leq  \E[U_P'(X^{0} - y)(a_{max} - m + y)].$$
It remains to deal with the second order term.  First note that : 
$$ U_P''(X^{a_\epsilon} - W_\epsilon) \leq  U_P''(X^{a_{max}} - m) \leq \sup_{B \in [b_1, b_2]} U_P''(X^{a_{max}} - m) \quad \mathbb{P}-a.s.,$$
where $S := \sup_{B \in [b_1, b_2]} U_P''(X^{a_{max}} - m) $ is a strictly negative real number and exists through the boundedness of $B$ and the DARA characteristic of $U_P$. Therefore~: 
\begin{align*}
\E[U_P''(X^{a_\epsilon} - W_\epsilon)(a_n - W_n + y)^2] &\leq  S \left( (a_{n} + y)^2 + \E[W_n]^2 - 2\E[W_n](a_n + y) \right)\\
& \leq S \left( (a_{max} + y)^2 + \E[W_n]^2 \right)  \underset{n \rightarrow + \infty}{\rightarrow} - \infty,
\end{align*}
and the conclusion is obtained :
$$ \E\left[U_P(X^{a_n} - W_n)\right] \underset{n \rightarrow + \infty}{\rightarrow} - \infty.$$

\end{proof}

\begin{Remark} The hypothesis $U''' \geq 0$ is satisfied for example by DARA (decreasing absolute risk aversion) utility functions. Empirical analysis mostly validates the DARA utility hypothesis (see for example \cite{Friend}), justifying the assumption of this theorem. 
\end{Remark}

\subsection{A general existence theorem}

The following Theorem gives existence of maximizers in general settings. 

\begin{Theorem}[Consequence to Theorems \ref{theo:exibound} and \ref{theo:exicoer}]
\label{theo:mainex}
Suppose that the Principal's utility function is strongly upper semi continuous, through Lemma \ref{lemma:usc} for example. The following maximizer existence results for the limited liability problems hold.  
\begin{enumerate}
\item Problem (\ref{eq:prob1}) has a maximizer. 
\item Problem (\ref{eq:prob2}) has a maximizer if both of the following two conditions hold : 
\begin{itemize}
\item[-] $ \text{supp}(B) \subset [b_1, b_2], \quad (b_1, b_2) \in \R^2. $
\item[-] $U_P''' \geq 0$. 
\end{itemize}
\end{enumerate}
\end{Theorem}
\begin{proof}
\begin{enumerate}
\item Through Lemma \ref{lem:coerpos}, the Principal's value function is coercive in $a$. Through positivity of any optimizer, optimization in $a$ may be restricted to some interval $[0, a_{max}]$. Then $C_m^M$ is both bounded in $W$ (by construction) and in $a$. Through Lemma \ref{lemma:closed}, $C_m^M$ is weakly closed. 
 Finally, through Lemma \ref{lemma:usc} and Lemma \ref{lemma:concave}, the Principal's value function is weakly upper semi-continuous. The result follows through Theorem \ref{theo:exibound}. 
\item  Through Lemma \ref{lem:coerpos}, the Principal's value function is coercive in $a$. Through Lemma \ref{lem:coer}, whose assumptions are satisfied, the value function is coercive in $W$. It is thus jointly coercive in $W$ and $a$. Indeed, let $(W_n, a_n)$ in $C_m$ be such that $$||(W_n,a_n)||_E   \underset{n \rightarrow + \infty}{\rightarrow} + \infty,$$
then it holds that that: 
\begin{itemize}
\item[-] either $a_n \underset{n \rightarrow + \infty}{\rightarrow}+ \infty$ and the Principal's value function is coercive through coercivity in $a$, \\
\item[-]
 or $0 \leq a_n \leq a_{max}$ where $a_{max}$ is some positive constant, and $$\E[W_n^2]   \underset{n \rightarrow + \infty}{\rightarrow}+ \infty.$$ The Principal's value function is then coercive through coercivity in $W$. 
\end{itemize}
 Now through Lemma \ref{lemma:closed}, $C_m^M$ is weakly closed. 
 Finally, through Lemma \ref{lemma:usc} and Lemma \ref{lemma:concave}, the Principal's value function is weakly upper semi-continuous. The result follows through Theorem \ref{theo:exicoer}. 
\end{enumerate}
\end{proof}

\begin{Remark}
Theorem \ref{theo:mainex} can also be applied in cases where the Principal's expected utility is directly weakly upper semi continuous. 
\end{Remark}

Theorem \ref{theo:mainex} uses assumptions on the Principal's utility function in order to ensure existence of solutions to the limited liability problem. However it does not use any assumptions on the Agent's utility function beyond it being increasing and concave. This allows for substantial freedom in modeling the Agent's utility. 

\begin{Example}
In the following examples, supposing that the Agent has any concave and increasing utility function, different limited liability existence results are provided. 
\begin{enumerate}
\item \textbf{Logarithmic utilities.} Let a Principal have extended logarithmic utility, as defined in (\ref{eq:utilog}). Suppose that the underlying uncertainty $B$ has a uniform distribution on $[-5 ; 5].$ Set $m=0$ and $M=2$. Then the contracting problem~: 
\begin{equation*}
\sup_{(W,a) \in C_{0}^2} \quad \E\left[U_P(X^a -W) \right], 
\end{equation*}
admits a maximizer. \\
\item \textbf{Mixed utilities.} Let a Principal have CARA utility. Suppose that the underlying uncertainty $B$ has a Gaussian distribution, i.e. $B \sim \mathcal{N}(0,1)$. Set $m=3$ and $M=10$. Then the contracting problem~ 
\begin{equation*}
\sup_{(W,a) \in C_3^{10}} \quad \E\left[U_P(X^a -W) \right], 
\end{equation*}
admits a maximizer. 
\item \textbf{Arctan utility}. Let a Principal have extended arctangent utility, as defined in (\ref{eq:utid}). Suppose that the underlying uncertainty $B$ has some unbounded distribution. Set $m=0$. Then the contracting problem~: 
\begin{equation*}
\sup_{(W,a) \in C_0} \quad \E\left[U_P(X^a -W) \right], 
\end{equation*}
admits a maximizer. 
\end{enumerate}
\end{Example}

The conditions given in Theorem \ref{theo:mainex} for existence in the solely lower bound case are more restrictive than in the double bound case. Indeed, for the double bound case existence is immediate through the structure of the problem. The only ingredient is a weakly or strongly upper semi continuous value function for the Principal. However in the lower bound case sufficient conditions for the Principal's value function to be coercive are also enforced. As a consequence, unbounded random variables do not fall into the setting and in particular the CARA utility / Gaussian case cannot be dealt with through the theorem.\\

The aim in the next part is to overcome this limitation and be able to analyze the standard CARA utility setting under a lower limited liability constraint.  

\subsection{Existence theorem for CARA utilities using a perturbation method}

 In the following, a specific existence theorem for the lower bound case with a CARA Principal and Agent and a general underlying uncertainty $B$ is provided. This extension is provided for completeness of this work. Indeed the CARA utility setting is widely used in the Principal-Agent literature yet existence for a CARA Principal under lower limited liability and with an unbounded $B$ is not guaranteed through Theorem \ref{theo:mainex}. This is overcome below with a perturbation method when both the Principal and the Agent are CARA. 

\begin{Theorem}[Existence theorem for lower limited liability and CARA utilities]
\label{theo:mainMCARA}
Suppose that $$U_P(x) := -e^{-\gamma_P x} \quad \text{and} \quad U_A(x) := -e^{-\gamma_A x},$$ for $\gamma_P > 0$ and $\gamma_A > 0$, then there exists a solution to Problem (\ref{eq:prob2}). 
\end{Theorem}
Before providing a proof of this Theorem, two intermediary results are given. The first one provides existence of results for a perturbed problem that introduces coercivity whilst also leading to tractable calculations. 
\begin{Lemma}
Set $$U_P(x) := -e^{-\gamma_P x} \quad \text{and} \quad U_A(x) := -e^{-\gamma_A x}.$$ Then there exists a solution to the following problem : 
\begin{equation}
\label{eq:prob2eps}
\sup_{(W,a) \in C_{m}} \quad F_\epsilon(W,a),
\end{equation}
where $F_\epsilon(W,a) = \E[U_P(X^a-W)] + \epsilon \E[U_P(- W)]. $
\end{Lemma}
\begin{proof}

 Through Lemma \ref{lemma:closed}, $C_m$ is weakly closed. Through Lemma \ref{lemma:usc} and Lemma \ref{lemma:concave}, the Principal's value function is weakly upper semi-continuous. Lemma \ref{lem:coerpos} shows $F_\epsilon$ is coercive in $a$. Finally, coercivity in $W$ is obtained through the construction of $F_\epsilon$. Indeed for any $x$ on $\R^+$,
\begin{equation}
\label{expx}
 e^{x} - x^2 \geq 0.
 \end{equation}

Therefore, consider a sequence $(W_n, a_n)_n$ in $C_m$ with $a_n \in [0, a_{max}]$ for any $n$ in $\mathbb{N}$. As $W  \geq m \; \mathbb{P}-a.s. $ (with $m \geq 0$), it holds that : 
\begin{equation}
\label{eq:coerc} \lim_{n \rightarrow + \infty} \E\left[W_n^2\right] = + \infty \; \Rightarrow  \; \lim_{n \rightarrow + \infty} \E\left[(\sqrt{\gamma_P} W_n)^2\right] = + \infty  \; \Rightarrow  \; \lim_{n \rightarrow + \infty} \E\left[e^{\gamma_P W_n}\right] = + \infty. 
\end{equation}
Now, as a consequence to the Participation Constraint, $ \E[W_n ] \geq y + \kappa(a_n)$. Through Jensen's inequality : 
 \begin{align*}
F_\epsilon(W_n,a_n) &=  -\E\left[e^{-\gamma_P(X^{a_n}-W_n)}\right] - \epsilon \E\left[e^{\gamma_P W_n}\right] \\
&\leq -e^{-\gamma_P(x_0 + \E[B] - y + a_n - \kappa(a_n))}   - \epsilon \E\left[e^{\gamma_P W_n}\right] \\
& \leq -e^{-\gamma_P(x_0 + \E[B] - y + \tilde{a} - \kappa(\tilde{a}))} - \epsilon \E\left[e^{\gamma_P W_n}\right],\\
\end{align*}
where $\tilde{a} := \frac{1}{K}.$ Therefore using (\ref{eq:coerc}) it holds that : 
$$ \lim_{n \rightarrow + \infty} \E\left[W_n^2\right]  \; \Rightarrow \; \lim_{n \rightarrow + \infty} F_\epsilon(W_n,a_n) = - \infty. $$
Combined with coercivity in $a$ gives that : 
$$ ||(W_n, a_n)||_E \rightarrow + \infty \; \Rightarrow \;  \lim_{n \rightarrow + \infty} F_\epsilon(W_n,a_n) = + \infty.$$
Existence of maximizers is then obtained through Theorem \ref{theo:exicoer}. \end{proof}
Denote as $(W_\epsilon, a_\epsilon)$ the optimal contract for perturbed Problem (\ref{eq:prob2eps}). As the perturbed value function is strongly convex, the contract is unique. 
\begin{Remark}
\label{rem:aint}
The perturbation term is only a function of $W$. The reasoning from Lemma \ref{lem:coerpos} and Remark \ref{rem:pos} still holds and thus there exists some $a_{max}$ such that  
$$ a_{\epsilon} \in [0, a_{max}],$$
for any $\epsilon > 0.$
\end{Remark}
\begin{Lemma}
\label{Lemma:boundW}
Let $(W_\epsilon, a_\epsilon)$ be an $\epsilon$-solution to Problem $ (\ref{eq:prob2eps})$ for a fixed $0 < \epsilon < 1$.
Then :
$$ \E\left[W_\epsilon^2\right] \leq C_{max},$$
for $C_{max}$ some constant. 
\end{Lemma}

\begin{proof}
Consider any $0 < \epsilon < 1.$
The unique $\epsilon$-solution $(W_\epsilon, a_\epsilon)$ can be characterized by applying Luenberger's generalized K.K.T. theorem as detailed in Section \ref{sec:CARA}. As such there exists some $\lambda_\epsilon \geq 0$ and $Z_\epsilon$ in $\mathcal{V}$ where
$$ \mathcal{V} := \left\{ \; X \in L^2(\Omega), \quad X \geq 0 \quad \mathbb{P}\text{-a.s.} \right\},  $$
such that :
$$ 
\begin{cases}
\gamma_P e^{-\gamma_P(X^{a_\epsilon}- W_\epsilon)} + \epsilon \gamma_P e^{\gamma_P W_\epsilon} - Z_\epsilon - \lambda_\epsilon \gamma_A e^{-\gamma_A (W_\epsilon - \kappa(a_\epsilon))} = 0 \quad &(1)\\
- \gamma_P \E[e^{-\gamma_P(X^{a_\epsilon}- W_\epsilon)}] + a_\epsilon K \lambda_\epsilon \E[e^{-\gamma_A (W_\epsilon - \kappa(a_\epsilon))}] = 0 \quad &(2)\\
\E[Z_\epsilon(W_\epsilon-m)] = 0 \quad \text{and} \quad \lambda_\epsilon\left( \E[e^{-\gamma_A (W_\epsilon - \kappa(a_\epsilon))}] - e^{-\gamma_A y}\right) = 0. \quad &(3) 
\end{cases}
$$
First suppose that $a_\epsilon = 0$ and/or $\lambda_\epsilon = 0.$ Then in (2) : $$ \E\left[e^{-\gamma_P(X^{a_\epsilon}- W_\epsilon)}\right] = 0, $$
which is absurd. In particular it must hold that $ \E[e^{-\gamma_P(X^{a_\epsilon}- W_\epsilon)}] \geq R$ where $- R < 0$ is the corresponding Risk-Sharing optimum (see for example \cite{Mart}). Therefore 
$$ \gamma_P \E[e^{-\gamma_P(X^{a_\epsilon}- W_\epsilon)}] = a_\epsilon K \lambda_\epsilon \E[e^{-\gamma_A (W_\epsilon - \kappa(a_\epsilon))}],$$
implying that for any fixed $\epsilon > 0$ : 
$$ a_\epsilon \lambda_\epsilon = \dfrac{\gamma_P \E[e^{-\gamma_P(X^{a_\epsilon}- W_\epsilon)}]}{K \E[e^{-\gamma_A (W_\epsilon - \kappa(a_\epsilon))}]} \geq \frac{R}{K e^{-\gamma_A y}} > 0.$$
Recall that by construction $\lambda_\epsilon \geq 0$. It follows that for any $\epsilon > 0,$ $a_\epsilon > 0$ and $\lambda_\epsilon > 0$ and as a consequence the Participation Constraint binds. Through (1) and as $(y, 0)$ belongs to $C_m$ it holds that :
$$ 0 \leq \lambda_\epsilon \leq \dfrac{\gamma_P \E[e^{-\gamma_P(X^0 - y)} + e^{\gamma_P y}] }{\gamma_A e^{-\gamma_A y}} := C,$$
and it follows that 
$$ a_{\epsilon} \geq \frac{R}{KCe^{-\gamma_A y}} > 0.$$
Now, as both $Z_\epsilon$ and $W_\epsilon - m$ are positive random variables, 
$$\E[Z_\epsilon(W_\epsilon-m)] = 0 \; \Rightarrow \;  Z_\epsilon(W_\epsilon - m) = 0 \; \mathbb{P}-a.s..$$
Therefore as soon as $Z_\epsilon > 0$, it must hold that $W_\epsilon = m$ and as a consequence : $$0 \leq Z_\epsilon \leq Z_\epsilon \textbf{1}_{W_\epsilon = m} \; \mathbb{P}-a.s..$$ Using (1) : 
\begin{align*}
0 \leq Z_\epsilon \leq Z_\epsilon\textbf{1}_{W_\epsilon = m} &\leq \gamma_P e^{-\gamma_P(X^{a_\epsilon}- m)} + \epsilon \gamma_P e^{\gamma_P m} \\
&\leq \gamma_P e^{-\gamma_P(X^{0}- m)} +  \gamma_P e^{\gamma_P m},  \; \mathbb{P}-a.s..
\end{align*}
Finally : 
\begin{align*}
0 \leq m \leq W_\epsilon &= \frac{1}{\gamma_P} \ln\left(\dfrac{Z_\epsilon + \lambda_\epsilon e^{-\gamma_A (W_\epsilon - \kappa(a_\epsilon))}}{\gamma_P (\epsilon + e^{-\gamma_P X^{a_\epsilon}})}\right)\\
&\leq \frac{1}{\gamma_P} \ln \left(\dfrac{Z_\epsilon + \lambda_\epsilon e^{-\gamma_A (W_\epsilon - \kappa(a_\epsilon))}}{\gamma_P e^{-\gamma_P X^{a_\epsilon}}} \right)\\
&\leq \frac{1}{\gamma_P} \ln\left( \dfrac{ \gamma_P e^{-\gamma_P(X^{0}- m)} +  \gamma_P e^{\gamma_P m} + C e^{-\gamma_A (m - \kappa(a_{max}))}}{ \gamma_P e^{-\gamma_P X^{a_{max}}}} \right),
\end{align*}
and for any $0 < \epsilon < 1$ it holds that : 
$$ \E\left[W_\epsilon^2\right] \leq \frac{1}{\gamma_P^2}  \E\left[\ln\left(   \dfrac{ \gamma_P e^{-\gamma_P(X^{0}- m)} +  \gamma_P e^{\gamma_P m} + C e^{-\gamma_A (m - \kappa(a_{max}))}}{ \gamma_P e^{-\gamma_P X^{a_{max}}}}\right)^2  \right] = C_{max}.$$

\end{proof}

With these tools in mind one may turn to the proof of Theorem \ref{theo:mainMCARA}.

\begin{proof}[Proof of Theorem  \ref{theo:mainMCARA}]
The aim is to consider the perturbed solutions $(W_\epsilon, a_\epsilon)$ and to analyze their behavior as $\epsilon$ goes to 0. As such consider some sequence $(\epsilon_n)_{n \in \mathbb{N}}$ with : 
 $$\begin{cases}
 \forall n \in \mathbb{N}, 0 < \epsilon_n < 1 \\
 \lim_{n \rightarrow + \infty} \epsilon_n = 0.
 \end{cases}
 $$  Through Remark \ref{rem:aint}, $a_{\epsilon_n}$ belongs to $[0, a_{max}]$ for any $n$. Also using Lemma \ref{Lemma:boundW} :
$$ \E\left[W_{\epsilon_n}^2\right] \leq C_{max},$$
for any $n \in \mathbb{N}$.
Using the Banach-Alaoglu theorem, one can find a subsequence of $\epsilon$ such that : 
$$ \epsilon_n \underset{n \rightarrow + \infty}{\longrightarrow} 0 \;,\; W_{\epsilon_n}  \underset{n \rightarrow + \infty}{\rightharpoonup} W^*\;,\; a_{\epsilon_n}  \underset{n \rightarrow + \infty}{\rightharpoonup} a^*,$$
where $(W^*, a^*)$ belongs to $C_m$ through Lemma \ref{lemma:closed}. Finally, recall that $(W,a) \mapsto \E[U_P(X^{a} - W)]$ is weakly upper-semi-continuous.  Thus for any $(W,a) \in C_m$, it holds that : 
\begin{align*}
\E[U_P(X^{a^*} - W^*)] &\geq \limsup_{n \rightarrow + \infty} \E[U_P(X^{a_{\epsilon_n}} - W_{\epsilon_n})] \\
 &\geq \limsup_{n \rightarrow + \infty} \E[U_P(X^{a_{\epsilon_n}} - W_{\epsilon_n})] + \epsilon_n \E[U_P(-W_{\epsilon_n})]\\
  &\geq \limsup_{n \rightarrow + \infty} \E[U_P(X^{a} - W)] + \epsilon_n \E[U_P(-W)] 
  \\&= \E[U_P(X^{a} - W)],\\
\end{align*}
$(W^*, a^*)$ is therefore an optimal contract for Problem (\ref{eq:prob2}).
\end{proof}

The perturbation method used in this proof ressembles that of the Tikhonov method for ill-posed problems (\cite{Tiho}), albeit for a different perturbation term. Indeed for each $\epsilon > 0,$ one computes $(W_\epsilon, a_\epsilon)$ by maximizing : 
$$ \E[U_P(X^a-W)] + \epsilon \E[U_P(- W)]$$
across $C_m$, which equates to minimizing : 
$$ \E[\exp(-\gamma_P(X^a-W))] + \epsilon \E[\exp(\gamma_PW)].$$
As such $(W_\epsilon, a_\epsilon)$ makes 
$ \E[\exp(-\gamma_P(X^{a_\epsilon}-{W_\epsilon}))] $ as small as possible without $ \E[\exp(\gamma_P W_\epsilon)]$ becoming too big, and through (\ref{expx}) without $\E\left[W_\epsilon^2\right]$ becoming too big. Lemma \ref{Lemma:boundW} formalizes this by showing that the $L^2-$norm of $W_\epsilon$ is upper-bounded for $\epsilon$ smaller than 1. 

\begin{Remark}
By combining a calculus of variation approach for an approximating problem with an exploitation of the necessary optimality conditions satisfied by the approximating optimum,  existence of maximizers for a general underlying production processes is proven. The CARA case may also be dealt with by fully computing the solution to the system of necessary optimality conditions. However doing so without specifying $B$ is difficult and this emphasizes the strength of the mixed approach. 
\end{Remark}

Note that this theorem also requires that the Agent's utility be CARA : this leads to an exploitable system of necessary conditions.  This contrasts with Theorem \ref{theo:mainex} where the Agent's utility can be different from that of the Principal.

\section{Characterization of the Limited Liability optima}
\label{sec:charac}

With the existence of solutions in mind, one may turn to their characterization. 

\subsection{Lagrangian methods in infinite dimensions}

 The theory used here for characterizing optimizers on convex subsets of Hilbert spaces is described in Luenberger's monograph "Optimization by Vector Space methods"  (\cite{Luenberger97}).  \\

%\begin{Theorem}[Luenberger's generalized K.K.T. theorem]
%\label{theorem:genKKT}
%Let $X$ be a vector space and $Z$ a normed space having a positive cone $P$. Assume that $P$ contains an interior point. \\
%Let $f$ be a Gateaux differentiable real-valued functional on $X$ and $G$ a Gateaux differentiable mapping from $X$ into $Z$. Assume that the Gateaux differentials are linear in their increments. Suppose that $x_0$ maximizes $f$ subject to $G(x) \geq \theta$ and that  $G(x_0) \geq \theta.$ Then there exists a $z_0^* \in Z^*$, $z_0^* \leq \theta$ such that the Lagrangian 
%$$ f(x) + \langle G(x), z_0^*\rangle$$ is stationary at $x_0$. Furthermore $\langle G(x_0), z_0^*\rangle = 0 $ (this is a slackness condition). 
%\end{Theorem}
%
%\begin{Remark}
%Through the stationary and slackness conditions one obtains the necessary optimality conditions for such a problem. Indeed, suppose that an optimal $x_0$ exists then it satisfies the following equalities  : 
%\begin{equation}
%\label{eq:char}
%\begin{cases}  f'(x_0) +  \langle G'(x_0), z_0^*\rangle = 0 \\
%\langle G(x_0), z_0^*\rangle = 0,
%\end{cases}
%\end{equation}
%with Gateaux derivation. As a consequence, when a maximizer exists, one may characterize it by solving (\ref{eq:char}).
%\end{Remark}

The aim is to characterize the limited liability optima in situations where its existence is ensured. Such situations are studied in Section \ref{sec:ex} and in the following maximizer existence is supposed as a preliminary as well as the Gâteaux-differentiability of the expected utility functions. \\

Let $\mathcal{P}$ be the positive cone of $E$ : 
$$ \mathcal{P} := \left\{ \; (W,a) \in E, \quad W \geq 0 \quad \mathbb{P}\text{-a.s.} \quad \text{and} \quad a \geq 0 \right\} = \mathcal{P}_W \times \mathbb{R^+},$$
where 
$$ \mathcal{P}_W := \left\{ \; W \in L^2(\Omega), \quad W \geq 0 \quad \mathbb{P}\text{-a.s.} \right\},  $$

The aim is to maximize $$(W,a) \mapsto \E\left[U_P(X^a-W)\right]$$
over a constraint set $C_m$ or $C_m^M$ which have either 2 or 3 constraints. In this setting one may identify as follows the relevant Lagrangian using the notation : 
$$ \tilde{U}_P = -U_P \quad \text{and} \quad  \tilde{U}_A = -U_A.$$

\begin{Definition}
Consider the mapping $$ \mathcal{L}_m^M : E \times \R^+ \times  \mathcal{P}_W \times  \mathcal{P}_W  \mapsto \R,$$ defined as : 
\begin{align*}
\mathcal{L}_m^M(W,a,\lambda, Z, Y) := F(a, W) + G_{PC}(a,W) + H_{LL}(W,Y,Z),\\
\end{align*}
where 
$$
\begin{cases}
F(a, W) := \E\left[\tilde{U}_P\left(X^a - W\right) \right], \\
 G_{PC}(a,W, \lambda) := \lambda \left( \E\left[ \tilde{U}_A(W-\kappa(a))\right]  - \tilde{U}_A(y)\right),\\
H_{LL}(W,Y,Z) :=  \E[Z(m-W)] +  \E[Y(W-M)].
\end{cases} $$

Then $\mathcal{L}_m^M$ is the Lagrangian for the double-sided limited liability problem.  The Lagrangian for the one-sided limited liability problem is obtained by formally setting $Y$ to 0 in $\mathcal{L}$.
\end{Definition}

The maximizer satisfies stationarity and slackness conditions given in the following Proposition. 
\begin{Proposition}[Necessary optimality conditions]
\label{prop:LLNO2}
Let $(W^*, a^*, \lambda^*, Z^*, Y^*)$ be a solution to Problem (\ref{eq:prob1}) along with the related Lagrange multipliers. Then the four following necessary optimality conditions must hold :  
\begin{enumerate}
\item $ \E\left[\tilde{U}_P'\left(X^{a^*} - W^*\right) \right] - \lambda^*\kappa'(a) \E\left[ \tilde{U}'_A(W^*-\kappa(a^*))\right]  = 0$\\
\item $   -\tilde{U}_P'(X^{a^*}- W^*) + \lambda^* \tilde{U}_A'(W^* - \kappa(a^*)) - Z^* + Y^* = 0 $\\
\item $\lambda^* \left(\E\left[U_A(W^*-\kappa(a^*))\right]  -U_A(y)\right)= 0$\\
\item $\E[Z^*(W^*-m)] = 0$ \; \text{and} \; $\E[Y^*(M - W^*)] = 0$. \\
\end{enumerate}
The optima $(W^*, a^*, \lambda^*, Z^*)$ for Problem  (\ref{eq:prob2}) satisfy the same optimality conditions where $Y$ is formally set to 0. 
\end{Proposition}

\subsection{Optimal contract characterization}

In the following the aim is to further characterize the optimal contracts. 

\subsubsection{One-sided limited liability}

Suppose that Problem (\ref{eq:prob2})  admits a solution. Let $(W^*, a^*, \lambda^*, Z^*)$ in $E \times \R^+ \times  \mathcal{P}_W $ be a solution to Problem (\ref{eq:prob2}) with the two related Lagrange multipliers. The following four equations are satisfied : 
\begin{equation}
\label{eq:NOC}
\begin{cases}
\E\left[\tilde{U}_P'\left(X^{a^*} - W^*\right) \right] - \lambda^*\kappa'(a) \E\left[ \tilde{U}'_A(W^*-\kappa(a^*))\right]  = 0\\
-\tilde{U}_P'(X^{a^*}- W^*) + \lambda^* \tilde{U}_A'(W^* - \kappa(a^*)) - Z^* = 0\\
\lambda^* \left(\E\left[U_A(W^*-\kappa(a^*))\right]  -U_A(y)\right)=0\\
\E[Z^*(W^*-m)] = 0.
\end{cases}
\end{equation}

\begin{Lemma}
\label{lemma:PC}
The optimal $\lambda^*$ satisfies :  
$$\lambda^* > 0,$$
meaning that the Participation Constraint binds at the optimum. 
\end{Lemma}
\begin{proof}
Using the second equality of (\ref{eq:NOC}) and as $Z^* \in \mathcal{P}_W$, it holds that : 
$$-\tilde{U}_P'(X^{a^*}- W^*) + \lambda^* \tilde{U}_A'(W^* - \kappa(a^*)) \geq 0,$$
implying that : 
$$ \lambda^* \geq \frac{\tilde{U}_P'(X^{a^*}- W^*) }{\tilde{U}_A'(W^* - \kappa(a^*)) } > 0.$$
\end{proof}
\begin{Lemma} 
\label{lemma:mW}The optimal $Z^*$ and $W^*$ satisfy : 
$$ Z^* (W^* - m)  = 0 \quad \mathbb{P}-a.s. \quad \text{and} \quad \mathbb{P}\left( W^* = m \right) < 1. $$
\end{Lemma}
\begin{proof}
By definition it holds that :
$$Z^* \geq 0 \quad \mathbb{P}-\text{a.s.} \quad \text{and} \quad W^* - m \geq 0 \quad \mathbb{P}-\text{a.s.}, $$
using the final equality of (\ref{eq:NOC}) : 
$$\E[Z^*(W^*-m)] = 0 \quad \Longrightarrow \quad Z^* (W^*-m) = 0 \quad \mathbb{P}-a.s.$$

For the second assertion, the first equality of (\ref{eq:NOC})  states that :
$$ \E\left[ \tilde{U}_A(W^* - \kappa(a^*))\right]  = \tilde{U}_A(y).$$
Now suppose for contradiction that 
$$ W^* = m \quad \mathbb{P}-a.s..$$
Then $a^*$ must verify $y=m-\kappa(a^*)$ for any $y \geq 0$. This is impossible as for any $x$ in $\R_+,$ $\kappa(x) \geq 0$, and a contradiction is reached.
\end{proof}

\begin{Theorem}[General characterization for one-sided limited liability]
\label{theo:gencarone}
Let $(W^*, a^*, \lambda^*, Z^*)$ be a solution to Problem (\ref{eq:prob2}). Then :
\begin{itemize}
\item[-]  $(W^*, a^*)$ saturates the participation constraint : $$ \E\left[ U_A(W^* - \kappa(a^*))\right]  = U_A(y).$$
\item[-](Borch rule for one-sided limited liability). On the event $\{Z^* =0\}$, the ratio of marginal utilities of the Principal and the Agent is constant :
$$ \lambda^* = \frac{\tilde{U}_P'(X^{a^*}-W^*)}{\tilde{U}_A'(W^*-\kappa(a^*))}  \textbf{1}_{Z^* = 0}.$$
\item[-] %When $U_A$ is uniformly integrable, 
The optimal action in the limited liability problem is greater than the optimal action in the standard risk-sharing problem, meaning that : 
$$a^* \geq \frac{1}{K}.$$
\end{itemize}
\end{Theorem}
\begin{proof} 
- The first statement is is a direct consequence to Lemma \ref{lemma:PC}. \\

\noindent-  Lemma \ref{lemma:mW} implies that : 
$$ \mathbb{P}\left( Z^* = 0 \right) > 0, $$
so setting oneself on the event $\{ Z^* = 0\}$ and using the second equation of Lemma \ref{eq:NOC} it holds that $$ - \tilde{U}_P'(X^{a^*}-W^*) + \lambda^* \tilde{U}_A'(W^*-\kappa(a^*)) = 0,$$
and the result follows. \\

\noindent- By definition, $Z^*$ is a positive random variable and therefore $\E[Z^*]\geq 0$. 

Using Equation 2 of Lemma \ref{eq:NOC} it holds that :  
\begin{align*}
\E[Z^*] &=  -\E\left[\tilde{U}_P'\left(X^{a^*} - W^*\right)\right] + \lambda^* \E\left[\tilde{U}_A'(W^*-\kappa(a^*)) \right] \\
&=  - \lambda^* \kappa'(a^*) \E\left[\tilde{U}_A'\left(W^* - \kappa(a^*)\right)\right]   + \lambda^* \E\left[U_A'(W^*-\kappa(a^*)) \right]\quad\\
&= - \lambda^*\E\left[\tilde{U}_A'\left(W^* - \kappa(a^*)\right)\right]   \left( \kappa'(a^*) - 1 \right),
\end{align*}
where $\lambda^* > 0$ and $\E\left[U_A'\left(W^* - \kappa(a^*)\right)\right]  < 0$. 
It must then hold that :
$$ \kappa'(a^*) - 1 \geq 0$$
which as $a^* \geq 0 $ (see Remark \ref{rem:pos}) is equates to $a^* \geq \frac{1}{K}$
where $\frac{1}{K}$ is the optimal action in the Risk Sharing problem. \\

\end{proof}

This series of results on the solution to Problem  (\ref{eq:prob2}) leads to some interesting remarks on the effect of the one-sided limited liability constraint.\\

First and just like in the standard Risk Sharing problem, the optimal contract must bind the participation constraint : this means that setting a lower bound for the wage does not have any effect on the Agent's expected utility. In fact one may deduce that $W^*$ and $a^*$ must compensate for the lower bound in order to still minimize the Agent's expected utility. So whilst the lower bound "improves" the Agent's situation by ensuring that the Agent's losses are limited, he is on average no better off.  Furthermore, the Agent's action is greater or equal to that of the optimal action in the Risk Sharing problem, and the action is therefore indeed able to compensate for the minimum wage constraint. What the Agent gains in diminished risk, he seems to lose in having to provide greater effort and one may wonder whether such a tradeoff is always worth it for the Agent. \\

Also, the limited liability constraint leads to a Borch rule  (\cite{Borch}, \cite{CZ}) holding only on part of the underlying probability space and characterizing the optimal wage as an "option" type wage. Indeed, introducing the limited liability constraint restricts the Borch rule to hold only on the specific event $\left\{ Z^* = 0 \right\}$ rather than everywhere. Note that on $\left\{ Z^* = 0 \right\}$, $W^*$ may be greater than or equal to $m$. However, the complement of $\left\{ Z^* = 0 \right\}$ is the event $\left\{ Z^* > 0 \right\}$ on which one must have : $W = m \quad  \mathbb{P}-a.s..$
So under limited liability constraints and working $\omega$ by $\omega$ , either a Borch rule holds or the optimal wage is worth $m$. It follows that optimal wage must be of option form, i.e. : 
 $$ W^* =  m\textbf{1}_{\tilde{W} < m} + \tilde{W}\textbf{1}_{\tilde{W} \geq m}  ,$$
with $\tilde{W}$ some random variable satisfying the Borch rule when valued above $m$. 

\subsubsection{Double-sided limited liability}

Suppose that Problem (\ref{eq:prob1}) admits a solution. Let $(W^*, a^*, \lambda^*, Z^*, Y^*)$ in $E \times \R^+ \times  \mathcal{P}_W \times \mathcal{P}_W $ be a solution to Problem (\ref{eq:prob1}) with the three related Lagrange multipliers. The following four equations are satisfied : 
\begin{equation}
\label{eq:NOC2}
\begin{cases}
\E\left[\tilde{U}_P'\left(X^{a^*} - W^*\right) \right] - \lambda^*\kappa'(a^*) \E\left[ \tilde{U}'_A(W^*-\kappa(a^*))\right]  = 0\\
-\tilde{U}_P'(X^{a^*}- W^*) + \lambda^* \tilde{U}_A'(W^* - \kappa(a^*)) - Z^*+ Y^*  = 0\\
\lambda^* \left(\E\left[U_A(W^*-\kappa(a^*))\right]  -U_A(y)\right)=0\\
\E[Z^*(W^*-m)] = 0 \quad \text{and} \quad \E[Y^*(M- W^*)] = 0.
\end{cases}
\end{equation}

\begin{Lemma}
\label{lem:eqmM}
Let $\omega$ be such that $Z^*(\omega) = Y^*(\omega)$. Then 
$$ Z^*(\omega) = 0 \quad \text{and} \quad Y^*(\omega) = 0. $$
\end{Lemma}
\begin{proof}

Let $\omega$ be such that $Z^*(\omega) = Y^*(\omega) = c$ for some non-negative $c$.
For reasons mirroring those of Lemma \ref{lemma:mW} it must hold that : 
$$ Z^* (\omega) W^*(\omega) = Z^*(\omega) m \quad \text{and} \quad
 Y^* (\omega) W^*(\omega) = Y^* (\omega) M,$$
which rewrites as 
$$
\begin{cases}
c W^*(\omega) = c m,\\
c W^*(\omega) = c M. 
 \end{cases}
 $$
 As $m \neq M$, this is only possible for $ c = 0. $
\end{proof}
\begin{Theorem}[General characterization for double-sided limited liability]
\label{theo:gencartwo}
Let\\ $(W^*, a^*, \lambda^*, Z^*, Y^*)$ be a solution to Problem (\ref{eq:prob1}). Then $(W^*, a^*, \lambda^*, Z^*, Y^*)$ satisfies the following statement :
\begin{itemize}
\item[-](Borch rule for double-sided limited liability). On the event $\{Z^* = Y^*\}$, the ratio of marginal utilities of the Principal and the Agent is constant :
$$ \lambda^* = \frac{\tilde{U}_P'(X^{a^*}-W^*)}{\tilde{U}_A'(W^*-\kappa(a^*))}  \textbf{1}_{Z^*=0 \quad \text{and}\quad Y^*=0}.$$
\end{itemize}
\end{Theorem}
\begin{proof}
The proof of Theorem \ref{theo:gencarone} may be adapted to this setting to show here that : 
$$ \lambda^* = \frac{\tilde{U}_P'(X^{a^*}-W^*)}{\tilde{U}_A'(W^*-\kappa(a^*))}  \textbf{1}_{Z^*=Y^*},$$
and the final result is obtained through Lemma \ref{lem:eqmM}. 
\end{proof}
Much of the interpretation of the one-sided limited liability characterization extends to double-sided limited liability. However, one may make some specific remarks. \\

First note that in this theorem there is no statement regarding the value of $a^*$. In fact, the analysis related to $a^*$ is slightly more subtle here. Indeed~:
\begin{align*}
\E[Z^* - Y^*] &=  -\E\left[\tilde{U}_P'\left(X^{a^*} - W^*\right)\right] + \lambda^* \E\left[\tilde{U}_A'(W^*-\kappa(a^*)) \right] \\
&=  - \lambda^* \kappa'(a^*) \E\left[\tilde{U}_A'\left(W^* - \kappa(a^*)\right)\right]   + \lambda^* \E\left[U_A'(W^*-\kappa(a^*)) \right] \\
&= - \lambda^*\E\left[\tilde{U}_A'\left(W^* - \kappa(a^*)\right)\right]   \left( \kappa'(a^*) - 1 \right),
\end{align*}
where $Z^*$ and $Y^*$ are two non-negative random variables, and the sign of the right hand side term is given by the sign of $\kappa'(a^*) - 1$. In other words, under double-sided limited liability constraints, the Agent's action can be either above or below that of the standard risk-sharing problem. This depends on the sign of $\E[Z^* - Y^*]$ which will be linked to the values of $m$ and $M$.\\

Also there is no statement about the Participation Constraint either, as one does not know the sign of $Z^* - Y^*$ to apply the same reasoning as in the one-sided case. In fact information on the probability of $Z^*$ being equal to $Y^*$ is lacking.\\

Finally one can once again specify the form of the optimal wage $W^*$ : 
 $$ W^* =  m\textbf{1}_{\tilde{W} < m} + \tilde{W}\textbf{1}_{m \leq \tilde{W} \leq M} + M\textbf{1}_{\tilde{W} > M}   ,$$
with $\tilde{W}$ some random variable satisfying the Borch rule when valued between $m$ and $M$. \\

The following example illustrates Theorem \ref{theo:gencartwo}. 

\begin{Example}\label{ex:log} \textbf{Logarithmic utilities.} Let a Principal and an Agent have extended logarithmic utility, as defined in (\ref{eq:utilog}), meaning that $U_P(x) = U_A(x) = U(x)$ where : 
$$U(x)= log(x)\textbf{1}_{x \geq 1} - \frac12(x^2 - 4x + 3)\textbf{1}_{x < 1},$$

Suppose that the underlying uncertainty $B$ has a distribution with compact support $[-b_{mix} ; b_{max}].$ Fix some lower bound $m \geq 0$, some upper bound $M$ and some reservation parameter $y \in [m;M]$. Then through Theorem \ref{theo:mainex} the contracting problem~: 
\begin{equation*}
\sup_{(W,a) \in C_{m}^M} \quad \E\left[U_P(X^a -W) \right], 
\end{equation*}
admits a maximizer. \\

Before providing characterization one must study the Gâteaux differentiability of the Principal and the Agent's expected utilities. Note that $U$ is a $C^{2}$ function across $\mathbb{R}$ with derivatives: 
$$U'(x)= \frac1x\textbf{1}_{x \geq 1} - \frac12(2x-4)\textbf{1}_{x < 1},$$
and
$$U''(x) = -\frac{1}{x^2}\textbf{1}_{x \geq 1} - \textbf{1}_{x < 1}. $$

For any $(W,a)$ in $C_m^M$, $W$ belongs to $[m, M]$ and $B-W$ belongs to $[b_{min}-M ; b_{max}-m].$ As a consequence, $U(X^a-W), U'(X^a-W), U''(X^a-W)$ as well as $U(W-\kappa(a)), U'(W-\kappa(a)), U''(W-\kappa(a))$ are bounded random variables. \\

Let any $H \in L^2(\Omega) $ such that $(W + \tau H,a)$ belongs to the convex set $C_m^M$ for $\tau \leq \epsilon$ for some $\epsilon > 0$. Now fix $\tau > 0 :$
$$ \E[U_P(X^a -(W + \tau H))] - \E[U_P(X^a-W)] = \E[U_P'(X^a - W)\tau H] + \frac{1}{2} \E[U_P''(X^a-W^{\tau H})\tau^2 H^2],$$ 
where $W^{\tau H}$ is a convex combination of $W$ and $W+ \tau H$. Boundedness of the terms ensures that the expectations exist and thus : 
$$\lim_{\tau \rightarrow 0} \frac{\E[U_P(X^a -(W + \tau H))] - \E[U_P(X^a-W)]}{\tau} = \E[U_P'(X^a - W)H] $$
and one may identify the Gâteaux derivative $U_P'(X^a - W)$. Similar reasoning can be performed for $U_A$.\\

Now, let $(W^*, a^*)$ in $E$ be a maximizer then one may characterize it. Indeed there exists $(\lambda^*, Z^*, Y^*)$ in $\R^+ \times \mathcal{P}_W \times \mathcal{P}_W$ such that :  

$$ a^* = {\kappa'}^{-1} \left( \frac{\E[Z^*-Y^*]}{\lambda^* \E[U'(W^*-\kappa(a^*))]}-1 \right),$$
and the wage :
$$W^* =  m\textbf{1}_{\tilde{W} < m} + \tilde{W}\textbf{1}_{m \leq \tilde{W} \leq M} + M\textbf{1}_{\tilde{W} > M} ,$$
satisfies the Borch rule : 
$$ \lambda^* = \frac{\tilde{U}'(X^{a^*}-W^*)}{\tilde{U}'(W^*-\kappa(a^*))}  \textbf{1}_{Z^*=Y^*},$$
with 
$$U'(x)= \frac1x\textbf{1}_{x \geq 1} - \frac12(2x-4)\textbf{1}_{x < 1}.$$
\end{Example}

\begin{Remark}
\label{rem:log}
The logarithm utilities example illustrates the wide scope of the previously proven results. Indeed this work provides general existence results for limited liability problems in which the underlying objects (eg. utility functions) lead to relatively unusable expressions. In fact with such utilities the method which involves proving simultaneously existence and characterizing the optima by solving the K.K.T. conditions would be difficult to use, and have to be done on more of a case by case basis. One can nevertheless also provide some general interesting information on the characterization. 
\end{Remark}

\section{Limited liability in the CARA utility case}
\label{sec:CARA}

The difficulties illustrated in Example \ref{ex:log} and Remark \ref{rem:log} show that economic analysis of the optimum in some settings is difficult. In the following the very tractable CARA case is discussed : $$ U_P(x) := - e^{-\gamma_P x} \quad \text{ and } \quad U_A(x) :=  - e^{-\gamma_A x}, $$
where $\gamma_P > 0$ and $\gamma_A > 0$ are two fixed risk aversion coefficients. In such a setting one may then further the characterizations given so-far.

\subsection{Optima characterization and analysis} Theorems \ref{theo:mainex}\ and \ref{theo:mainMCARA} guarantee existence of both a one-sided and double-sided limited liability optimum.  In the following Proposition closed form expressions of these optima is provided.
 \begin{Proposition}[Limited liability contracts in a CARA setting]
\label{prop:optana}
\begin{itemize}
\item[1.] Let \\$(W^*, a^*, \lambda^*, Z^*)$ be the optimum for one-sided limited liability. Then set :
$$ \tilde{W} := \frac{\gamma_P}{\gamma_P + \gamma_A} X^{a^*} + \frac{\gamma_A}{\gamma_P + \gamma_A} \kappa(a^*) + \frac{1}{\gamma_P + \gamma_A}\ln\left(\frac{\gamma_A \lambda^*}{\gamma_P}\right). $$
The optimal wage $W^*$ is of the form : 
$$ W^* =  m \textbf{1}_{\tilde{W} < m} + \tilde{W} \textbf{1}_{\tilde{W} \geq m}$$
and the optimal action $a^*$ is of the form : 
$$ a^* = \kappa^{-1}\left( 1 + e^{\gamma_A y}\frac{\E[Z^*]}{\lambda^* \gamma_A}\right). \\$$
\item[2.] Let $(W^*, a^*, \lambda^*, Z^*, Y^*)$ be the optimum for double-sided limited liability. Then set :
$$ \tilde{W} := \frac{\gamma_P}{\gamma_P + \gamma_A} X^{a^*} + \frac{\gamma_A}{\gamma_P + \gamma_A} \kappa(a^*) + \frac{1}{\gamma_P + \gamma_A}\ln\left(\frac{\gamma_A \lambda^*}{\gamma_P}\right).$$
The optimal wage $W^*$ is of the form : 
$$ W^* =  m \textbf{1}_{\tilde{W} < m} + \tilde{W} \textbf{1}_{m \leq \tilde{W} \leq M} + M \textbf{1}_{\tilde{W} > M}$$
and the optimal action $a^*$ is of the form : 
$$ a^* = \kappa^{-1}\left( 1 + e^{\gamma_A y}\frac{\E[Z^* - Y^*]}{\lambda^* \gamma_A}\right). \\$$
Note that it holds that : 
$$ \E[Z^* - Y^*] \geq \lambda^* \gamma_A e^{-\gamma_A y} \left(\kappa'(0)-1 \right) \geq  - \lambda^*  \gamma_A e^{-\gamma_A y},$$
thus ensuring that $a^*$ is non-negative.
\item[3.] In both cases the participation constraint is saturated at the optimum. 
\end{itemize}
\end{Proposition}
\begin{proof}
This Proposition is a direct consequence of Proposition \ref{prop:LLNO2} written with CARA utilities (G\^ateaux differentiability of the utilities is a known result). To see that the Participation Constraint must be saturated, use the first optimality condition. In both cases~: 
$$ -\gamma_P \E\left[e^{-\gamma_P(X^{a^*}-W^*)}\right] + \lambda^* \gamma_A \kappa'(a^*) \E\left[e^{-\gamma_A(W^* - \kappa(a^*))}\right] = 0,$$
then either $a^* = 0$ and 
$$ \E\left[e^{-\gamma_P(X^{a^*}-W^*)}\right]  = 0,$$ 
which is absurd, or $a^* > 0$ : 
$$ \lambda^* = \frac{\gamma_P \E\left[e^{-\gamma_P(X^{a^*}-W^*)}\right] }{\gamma_A \kappa'(a^*) \E\left[e^{-\gamma_A(W^* - \kappa(a^*))}\right] } > 0.$$
It follows that $\lambda^* > 0$ and the Participation Constraint is saturated. The characterizations follow. 
\end{proof}

In both cases the optimal wage is familiar. Indeed it is a truncated variant of the form of optimal wage in the Risk-Sharing case as established for example in \cite{CZ}.

\subsection{Limited liability in a gaussian setting}

As a conclusion, a discussion on the very problem that first motivated this work is provided~: enforcing wage positivity for an underlying gaussian production process with CARA utilities. As such from now on set : $$B \sim \mathcal{N}(0,1).$$ The optimal wage and action pair in the Risk-Sharing setting, i.e. Problem (\ref{eq:pbRS})-(\ref{eq:pbcons}), is then of the form (see for example \cite{CZ}) : 
$$ W^* = \frac{\gamma_P}{\gamma_P + \gamma_A} X^{a^*} + \beta \quad \text{and} \quad a^* = \frac1K,$$
for $\beta$ some constant ensuring that the Participation Constraint binds. The analysis performed in this document allows a comparison of this result to the optima of Problem (\ref{eq:pbRS})-(\ref{eq:pbcons}) with the additional constraint $W \geq 0 \quad \mathbb{P}-a.s..$
Solution existence is then guaranteed by Theorem \ref{theo:mainMCARA} and characterization of the optima is given in Proposition \ref{prop:optana}. In the following, analysis of the optimal positive wage and action is performed along with a comparison to the related Risk-Sharing optima. This is done in two settings that stand out as most interesting. \\

Suppose that one of the following statements holds :
\begin{itemize}
\item[(A1)] the value of $y$ is greater than or equal to that of $x_0$,
\item[(A2)]  or the value of $\gamma_P$ is of similar order to or smaller than that of $\gamma_A$,
\end{itemize}

then the following observations may be made on the optimum. \\

\textit{1. The value of $\E[Z^*]$ is close to 0.}\\

This can be touched on analytically. Indeed, $Z^*$ is positive only when $X^{a^*} \leq \frac{-\beta_{a^*}}{\rho}$ which equates to $$B \leq -x_0 - a^* - \frac{\beta_{a^*}}{\rho}. $$
In practice the right-hand term here is often negative, and this inequality is verified quite rarely. When it is verified, the values of $Z^*$ generally remain close to $0$ so the expectation of the random variable is non-negative (as required) but low. \\

\textit{2. The optimal action is close to $\frac{1}{K}$.} \\

This can be expected from the expression of $a^*$, as $\E[Z^*]$ generally close to 0 and not compensated by $\lambda^*$.  When this is the case, the optimal intercept $\beta$ also tends to the optimal $\beta$ from the Risk-Sharing problem. \\

\textit{3. The Principal's expected utility is then drastically reduced.}\\

This is also unsurprising as the Agent's expected utility remains the same through the binding participation constraint, whilst the form of wage is favorable to the Agent and not to the Principal. In fact as $a^*_{LL} \approx a^*_{RS}$ and $\beta^*_{LL} \approx \beta^*_{RS}$, it holds that that : 
$$ \E\left[U_P\left(X^{a^*_{LL}} - W^*_{LL}\right)\right] \approx \E\left[U_P\left(X^{a^*_{RS}} - W^*_{RS}\right)\mathbbm{1}_{W^*_{RS}}\right]  $$

Figures \ref{fig:test7} to \ref{fig:test9} illustrate this behavior.  On each graph, the optimal Limited Liability wage is represented by the red line and the optimal Risk-Sharing wage under the same parameters represented by the blue dashes. They are both plotted as functions of the outcomes of $X^{a^*}$ on the real line.

\begin{figure}[h!]
\centering
\begin{minipage}{0.48\textwidth}
  \centering
  \includegraphics[width=.75\linewidth]{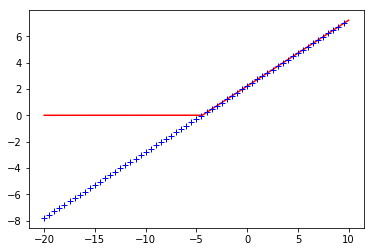}
  \captionof{figure}{$\gamma_P = 0.2$, $\gamma_A=0.2$, $K=2$ $x_0=1$, $y=1$. Optimum : $a^*_{RS}=0.5$, $\beta^*_{RS} = 0.525$,  $a^*_{LL}=0.5$, $\beta^*_{LL} = 0.524$}
  \label{fig:test7}
\end{minipage}%
~
~
\begin{minipage}{0.48\textwidth}
  \centering
  \includegraphics[width=.75\linewidth]{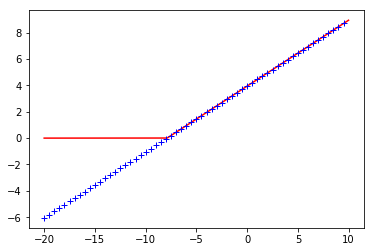}
  \captionof{figure}{$\gamma_P = 0.2$, $\gamma_A=0.2$, $K=2$ $x_0=1$, $y=3$. Optimum : $a^*_{RS}=0.5$, $\beta^*_{RS} = 0.525$,  $a^*_{LL}=0.5$, $\beta^*_{LL} = 0.525$}
  \label{fig:test8}
\end{minipage}
\end{figure}

\begin{figure}[h!]
\centering
\includegraphics[width=.37\linewidth]{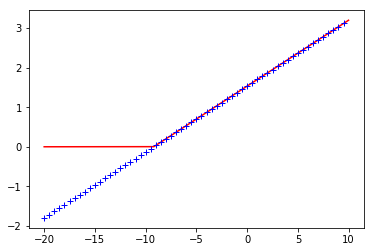}
 \caption{$\gamma_P = 0.2$, $\gamma_A=1$, $K=2$ $x_0=1$, $y=1$. Optimum : $a^*_{RS}=0.5$, $\beta^*_{RS} = 1.014$, $a^*_{LL}=0.5$, $\beta^*_{LL} = 1.014$. }
 \label{fig:test9}
\end{figure}

\begin{Remark}
\label{rem:concRA}
This analysis and these illustrations lead to the conclusion that a CARA Principal who is at most as risk-averse as the Agent he is contracting with will enforce limited liability by taking the positive part of a wage that is very close to the Risk-Sharing wage and ask the Agent to provide very similar effort to the Risk-Sharing effort. The Agent's expected utility stays the same as under Risk-Sharing, through the saturated Participation Constraint. The Principal's utility is however reduced. 
\end{Remark}

Now suppose that (A1) and (A2) do not hold but that one of the following holds :
\begin{itemize}
\item[(A3)] the value of $y$ is less than that of $x_0$,
\item[(A4)]  or the value of $\gamma_P$ is significantly greater than that of $\gamma_A$,
\end{itemize}

then the following observations may be made on the optimum. 

\textit{1. The Agent's action is significantly greater than $\frac{1}{K}$.}\\

\textit{2. The intercept of the linear part of the wage is not longer close to the Risk-Sharing intercept.}\\

This may be visualized in Figures \ref{fig:test10} and \ref{fig:test11} where the optimal Limited Liability wage is again represented by the red line and the optimal Risk-Sharing wage under the same parameters represented by the blue dashes.
\begin{figure}[h!]
\centering
\begin{minipage}{0.48\textwidth}
  \centering
  \includegraphics[width=.8\linewidth]{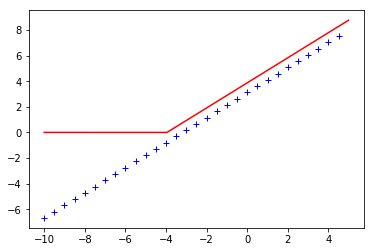}
  \captionof{figure}{$\gamma_P = 5$, $\gamma_A=0.1$, $K=2$ $x_0=1$, $y=0.5$. Optimum : $a^*_{RS}=0.5$, $\beta^*_{RS} = -0.673$,  $a^*_{LL}=1.358$, $\beta^*_{LL} = 0.077$}
  \label{fig:test10}
\end{minipage}%
~
~
\begin{minipage}{0.48\textwidth}
  \centering
  \includegraphics[width=.8\linewidth]{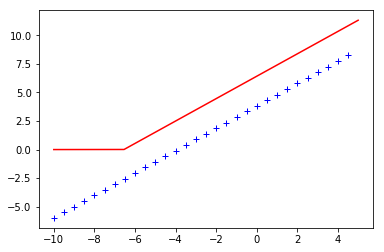}
  \captionof{figure}{$\gamma_P = 5$, $\gamma_A=0.1$, $K=2$ $x_0=5$, $y=0.5$. Optimum : $a^*_{RS}=0.5$, $\beta^*_{RS} = -4.594$,  $a^*_{LL}=2.10$, $\beta^*_{LL} = -2$}
  \label{fig:test11}
\end{minipage}
\end{figure}

\begin{Remark}
Remark \ref{rem:concRA} may be completed with analysis of a CARA Principal who is much more risk-averse than the Agent. In this case the Agent's expected utility remains that same as in the Risk-Sharing case. However the Principal asks more of the Agent by asking him to provide a substantially higher effort. The intercept of the linear part of the wage is then calibrated to saturate the Limited Liability constraint. 
\end{Remark}
\newpage

\section{Conclusion}

In this paper the problem of enforcing constant liability constraints in the benchmark Principal-Agent problem in the context of a single-period setting is tackled. General existence results are provided through a calculus of variations approach. For completeness, existence in the case of CARA utility and an unbounded production process with a lower bound constraint is proven through a perturbation approach. Characterizations of the optima and a Borch rule for the optimal wage are then provided.  In a CARA utility setting the optima is of closed-form and it is shown that it is closely linked to the standard Risk-Sharing optimum. In fact, in the Gaussian case, the optima is often very close to the Risk-Sharing optimum with a difference appearing when the Principal becomes highly risk-averse. \section*{Acknowledgements}
The author wishes to thank her PhD advisor Stéphane Villeneuve for the introduction to Principal-Agent problems and for many interesting discussions on the subject. She is also extremely grateful to Fr\'ed\'eric de Gournay for numerous discussions on infinite-dimensional optimization. Finally she wishes to thank the ANR Pacman for financial support.

\end{document}